\documentclass[12pt]{amsart}
\usepackage{graphicx}
\numberwithin{equation}{section} 

\newtheorem{thm}{Theorem}[section]
\newtheorem{cor}[thm]{Corollary}

\newtheorem{prop}[thm]{Proposition}

\theoremstyle{definition}
\newtheorem{defn}[thm]{Definition}
\theoremstyle{remark}
\newtheorem{rem}[thm]{Remark}

\newcommand{\bea}{\begin{eqnarray}}
\newcommand{\eea}{\end{eqnarray}}
\newcommand{\ba}{\begin{array}}
\newcommand{\ea}{\end{array}}
\newcommand{\bc}{\begin{center}}
\newcommand{\ec}{\end{center}}
\newcommand{\be}{\begin{equation}}
\newcommand{\ee}{\end{equation}}

\def\bn{{\mathbb N}}
\def\bz{{\mathbb Z}}

\def\d{\delta}

\def\i{\varepsilon}

\def\cf{{\mathcal F}}
\def\cm{{\mathcal M}}

\def\a{\alpha}

\def\b{\beta}
\def\m{\mu}
\def\n{\nu}
\def\l{\lambda}

\def\g{\gamma}

\begin{document}

\title[On $L_1$-weak ergodicity]
{On $L_1$-Weak Ergodicity of nonhomogeneous discrete Markov
processes and its applications}

\author{Farrukh Mukhamedov}
\address{Farrukh Mukhamedov\\
 Department of Computational \& Theoretical Sciences\\
Faculty of Science, International Islamic University Malaysia\\
P.O. Box, 141, 25710, Kuantan\\
Pahang, Malaysia} \email{{\tt far75m@yandex.ru} {\tt
farrukh\_m@iium.edu.my}}

\begin{abstract}
In the present paper we investigate the $L_1$-weak ergodicity of
nonhomogeneous discrete Markov processes with general state spaces.
Note that the $L_1$-weak ergodicity is weaker than well-known weak
ergodicity. We provide a necessary and sufficient condition for such
processes to satisfy the $L_1$-weak ergodicity. Moreover, we apply
the obtained results to establish $L_1$-weak ergodicity of discrete
time quadratic stochastic processes. As an application of the main
result, certain concrete examples are also provided. \vskip 0.3cm
\noindent


\end{abstract}

\maketitle

\section{Introduction}

Markov processes with general state space have become a subject of
interest due to their applications in many branches of mathematics
and natural sciences. One of the important notions in these studies
is ergodicity of Markov processes, i.e. the tendency for a chain to
'forget' the distant past. In many cases, a huge number of
investigations were devoted to such processes with countable state
space (see for example, \cite{D}-\cite{JI},\cite{HB},\cite{T}). For
nonhomogeneous Markov processes with countable state space,
investigation of the general conditions of weak ergodicity leads to
the definition of a special subclass of regular matrices. In many
papers (see for example, \cite{I,MI,P,T}) the weak ergodicity of
nonhomogeneous Markov process are given in terms of Dobrushin's
ergodicity coefficient \cite{D}. In general case, one may consider
several kinds of convergence \cite{MC}. In \cite{ZI} some sufficient
conditions for weak and strong ergodicity of nonhomogeneous Markov
processes are given and estimates of the rate of convergence are
proved.  Lots of papers were devoted to the investigation of
ergodicity of nonhomogeneous Markov chains (see, for example
\cite{D}-\cite{JI},\cite{T},\cite{Z}).

In the present paper we are going to investigate the $L_1$-weak
ergodicity of nonhomogeneous discrete Markov processes, in general
state spaces, without using Dobrushin's ergodicity coefficient. Note
that the $L_1$-weak ergodicity is weaker than usual weak ergodicity
(see next section). We shall provide necessary and sufficient
conditions for such processes to satisfy the $L_1$-weak ergodicity.
As application of the main result, certain concrete examples are
provided. Note that in \cite{D} similar conditions were found for
nonhomogeneous Markov processes to satisfy weak ergodicity. It is
worth to mention that in \cite{SG}  a necessary and sufficient
condition was found for homogeneous Markov processes to satisfy
$L_1$-ergodicity. Our condition recovers the mentioned condition
when the processes is homogeneous. Moreover, we will provide some
applications of the main result to $L_1$-weak ergodicity of discrete
quadratic stochastic processes which improves the result of
\cite{SGa}. Note that such processes relate to quadratic operators
\cite{K} as Markov processes relate to linear operators. For the
recent review on quadratic operator we refer to \cite{GMR}.

\section{$L_1$-Weak ergodicity}

Let $(X,\mathcal{F},\m)$ be a probability space. In what follows, we
consider the standard $L^1(X,\mathcal{F},\m)$ and
$L^{\infty}(X,\mathcal{F},\m)$ spaces. Note that
$L^1(X,\mathcal{F},\mu)$  can be identified with the space of finite
signed measures on $X$ absolutely continuous with respect to $\m$.
By $\frak{M}$ we denote the set of all probability measures on $X$
which are absolutely continuous w.r.t. $\m$.
 Recall that transition probabilities $P^{[k,m]}(x,A)$, $x\in X$, $A\in\cf$ ($k,n\in\bz_+$)
 form \textit{a nonhomogeneous discrete Markov process (NHDMP)} iff the
following conditions are satisfied:
\begin{enumerate}
\item[1.] for each $k,n$ the function of two variables $P^{[k,n]}(x,A)$ is
a Markov kernel, and it is $\m$-measurable, i.e. $\m(A)=0$ implies
$P^{[k,n]}(x,A)=0$ a.e. on $X$.

\item[2.] one has Kolmogorov-Chapman equation: for every $k\leq
m\leq n$
\begin{equation}\label{KC}
P^{[k,n]}(x,A)=\int P^{[k,m]}(x,dy)P^{[m,n]}(y,A).
\end{equation}
\end{enumerate}

In the sequel, we will deal with $\m$-measurable NHDMP. In this
case, for each $k,n$ such one can define a positive linear
contraction operator on $L^1$ (resp. $L^{\infty}$) denoted by
$P^{[k,n]}_*$ (resp. $P^{[k,n]}$). Namely,

\begin{eqnarray}\label{P*}
&& (P^{[k,n]}_*\n)(A)=\int P^{[k,n]}(x,A)d\n (x), \ \ \ \ \n\in L^1\\[2mm]
\label{P} &&(P^{[k,n]}f)(x)=\int P^{[k,n]}(x,dy)f(y), \ \ \ f\in
L^\infty.
\end{eqnarray}
 It is clear that
$\|P^{[k,n]}_*\n\|_1=\|\n\|_1$ for every positive measure $\n\in
L^1$.

From \eqref{P*} it follows that \eqref{KC} can be rewritten as
follows
$$
P^{[k,n]}_*=P^{[m,n]}_*P^{[k,m]}_*
$$
where $k\leq m\leq n$.

Recall that if for a NHDMP $P^{[k,n]}(x,A)$ one has
$P^{[k,n]}_*=\big(P^{[0,1]}_*\big)^{n-k}$, then such process becomes
{\it homogeneous}, and therefore, it is denoted by $P^{n}(x,A)$.

\begin{defn} A NHDMP $P^{[k,n]}(x,A)$ is said to satisfy
\begin{enumerate}
\item[(i)] the \textit{weak ergodicity}
if for any $k\in\bz_+$ one has
$$\lim_{n\rightarrow\infty}\sup_{x,y\in X}\|P^{[k,n]}(x,\cdot)-P^{[k,n]}(y,\cdot)\|_1=0;$$

\item[(ii)] the \textit{$L_1$-weak ergodicity}
if for any probability measures $\l,\nu\in\frak{M}$ and $k\in\bz_+$
one has
$$\lim_{n\rightarrow\infty}\|P^{[k,n]}_*\l-P^{[k,n]}_*\nu\|_1=0;$$

\item[(iii)] the \textit{strong ergodicity} if there exists a probability
measure $\m_1$ such that for every $k\in\bz_+$ one has
$$
\lim_{n\rightarrow\infty}\sup_{x\in
X}\|P^{[k,n]}(x,\cdot)-\m_1\|_1\to 0;
$$
\item[(iv)] the \textit{$L_1$-strong ergodicity} if there exists a probability
measure $\m_1$ such that for every $k\in\bz_+$ and $\l\in\frak{M}$
one has
$$
\lim_{n\rightarrow\infty}\|P^{[k,n]}_*\l-\m_1\|_1\to 0.
$$
\end{enumerate}
\end{defn}

One can see that the weak (resp. strong) ergodicity implies the
$L_1$-weak (resp. $L_1$-strong) ergodicity. Indeed, let us consider
the following example.

{\sc Example.} Let $X=\{1,2,3,4\}$ and $\m=(1/2,1/2,0,0)$. In this
case, the set $\frak{M}$ coincides with $\{(\a,1-\a,0,0): \ \
\a\in[0,1]\}$. Consider stochastic matrix
\begin{equation*}
\mathbb{P}=\left(
\begin{array}{llll}
p & q & 0 & 0 \\
q & p & 0 & 0\\
0 & 0 & 1 & 0\\
0 & 0 & 0 & 1\\
\end{array}
\right), \ \ \ p\in(0,1), \ p+q=1,
\end{equation*}
which is clearly $\m$-measurable. One can check that for any
$\l\in\frak{M}$ (i.e. $\l=(\a,1-\a,0,0)$, $\a\in[0,1]$) we have
$$
\mathbb{P}^n_*\l\to (1/2,1/2,0,0) \ \ \textrm{as} \ \ n\to\infty,
$$
this means $\mathbb{P}$ satisfies the $L_1$-strong ergodicity. On
the other hand, the matrix $\mathbb{P}$ has another two invariant
measures, i.e.
$$
\m_1=(0,0,1,0), \ \ \m_2=(0,0,0,1)
$$
which implies that $\mathbb{P}$ is not strong ergodic.\\

Therefore, it is natural to find certain necessary and sufficient
conditions for the satisfaction $L_1$-weak ergodicity of NHDMP. So,
in the paper we will deal with $L_1$-weak ergodicity. Note that
historically, one of the most significant conditions for the weak
ergodicity is the Doeblin's Condition (for homogeneous Markov
processes), which is formulated as follows: there exist a
probability measure $\n$, an integer $n_0\in\bn$ and constants
$0<\i<1$, $\d>0$ such that for every $A\in\cf$ if $\n(A)>\i$ then
$$
\inf_{x\in X}P^{n_0}(x,A)\geq\d.
$$
Such a condition does not imply either the aperiodicity or the
ergodicity of the process. In \cite{N} the aperiodicity is studied
by minorization type conditions, i.e. there exist a non-trivial
positive measure $\l$ and $n_0\in\bn$ such that
$$
P^{n_0}(x,A)\geq \l(A), \ \ \forall x\in X, \ \forall A\in\cf.
$$

But this condition is not sufficient for the strong ergodicity. In
\cite{SG} it was introduced a variation of the above condition, i.e.
Condition ($C_0$): there exists a non-trivial positive measure
$\m_0\in L^1$, $\|\m_0\|_1\neq 0$, and for every $\l\in\frak{M}$ one
can find a sequence $\{X_n\}\subset\cf$ with $\m(X\setminus X_n)\to
0$, as $n\to\infty$, and $n_0\in\bn$ such that for all $n\geq n_0$
one has\footnote{Here and in what follows, a given $B\in\cf$ the
measure $\m 1_{B}$ is defined by $\m 1_{B}(Y)=\m(Y\cap B)$ for any
$Y\in\cf$.}
\begin{equation}\label{AA}
P^{n}_*\l\geq\m_01_{X_n},
\end{equation}
where $1_Y$ stands for the indicator function of a set $Y$. It has
been proved that such a condition is necessary and sufficient for
the $L_1$-strong ergodicity of homogeneous processes. In the present
paper we shall introduce a simple variation of the above condition
($C_0$) for NHDMP, and prove that the introduced condition is a
necessary and sufficient for the $L_1$-weak ergodicity. Note that an
other direction of variation of the Doeblin's Condition has been
studied in \cite{DP}.

\section{Main results}

In this section we are going to introduce a simple variation of
condition ($C_0$).

\begin{defn}
We say that a NHDMP $P^{[k,n]}(x,A)$ given on $(X,\cf,\m)$ satisfies
\textit{condition ($C_1$)} if for each $k\in\bz_+$ there exist a
positive measure $\m_k\in L^1$, $\|\m_k\|_1\neq 0$,  and for every
$\d>0$ and $\l,\n\in\frak{M}$ one can find sets $X_k,Y_k\in \cf$
with $\m(X\setminus X_k)<\d$, $\m(X\setminus Y_k)<\d$ and an integer
$n_k\in\bn$ such that
\begin{equation}\label{1C}
P^{[k,k+n_k]}_*\l\geq\m_k1_{X_k}, \ \
P^{[k,k+n_k]}_*\n\geq\m_k1_{Y_k},
\end{equation}
\end{defn}
here as before $1_Y$ stands for the indicator function of a set $Y$.

\begin{rem} In \eqref{C},\eqref{1C} without loss of generality we may assume
that $\|\m_k\|_1<1/2$, otherwise we will replace $\m_k$ with
$\m_k'=\m_k/2$.
\end{rem}

\begin{prop}\label{EP1} Let a NHDMP $P^{[k,n]}(x,A)$ given on $(X,\cf,\m)$.
Then for the following assertions
\begin{enumerate}
\item[(i)] $P^{[k,n]}(x,A)$ satisfies condition ($C_1$);

\item[(ii)] for any $\l,\n\in\frak{M}$ and $k\in\bz_+$ there is a sequence
$\{n_i\}$ such that for all $n\geq K_\ell:=\sum_{i=1}^\ell n_i$
($K_0=k$) one has
\begin{equation}\label{l-m}
\|P^{[k,n]}_*\l-P^{[k,n]}_*\n\|_1=\bigg(\prod_{i=1}^{\ell}\g_{i}\bigg)
\|P^{[K_{\ell},n]}_*\l_\ell-P^{[K_\ell,n]}_*\n_\ell\|_1,
\end{equation}
where $\l_\ell,\n_\ell\in\frak{M}$, and
\begin{equation}\label{gg}
\frac{1}{2}\leq \g_{i}\leq 1-\frac{\|\m_{K_{i-1}}\|_1}{2}, \ \
i=1,\dots,\ell.
\end{equation}
\end{enumerate}
the implication hold true: (i)$\Rightarrow$(ii). \end{prop}

\begin{proof} Take any $\l,\n\in\frak{M}$ and fix
$k\in\bz_+$. Let us prove \eqref{l-m} by induction. Due to condition
($C_1$) there is a measure $\m_k$. Then according to absolute
continuity of Lebesgue integral, there is $\d_1>0$ such that for any
$Z\in\cf$ with $\m(Z)<2\d_1$ one has
\begin{equation}\label{Z}
\int\m_k1_{Z}d\m<\frac{\|\m_k\|_1}{2}.
\end{equation}
Now again due to condition ($C_1$) there are $X_1, Y_1\subset\cf$
and $n_1\in\bn$ such that one has $\max\{\m(X\setminus
X_1),\mu(X\setminus Y_1)\}<\d$ and
\begin{equation}\label{2C}
P^{[k,k+n_1]}_*\l\geq\m_k1_{X_1}, \ \
P^{[k,k+n_1]}_*\n\geq\m_k1_{Y_{1}}.
\end{equation}
Denoting $Z_{1}=X_{1}\cap Y_{1}$, one has $\m(X\setminus
Z_{1})<2\d$, and from \eqref{2C} we find
\begin{equation}\label{3C}
P^{[k,k+n_1]}_*\l\geq\m_k1_{Z_{1}}, \ \
P^{[k,k+n_1]}_*\n\geq\m_k1_{Z_{1}}.
\end{equation}

It follows from \eqref{3C} that
\begin{eqnarray}\label{l-z}
\|P^{[k,k+n_1]}_*\l-\m_k1_{Z_{1}}\|_1&=&\int\big(P^{[k,k+n_1]}_*\l-\m_k1_{Z_{1}}\big)d\m\nonumber\\[2mm]
&=&\int P^{[k,k+n_1]}_*\l d\m-\int\m_01_{Z_{1}}d\m\nonumber\\[2mm]
&=&1-\int\m_01_{Z_{1}}d\m\nonumber\\[2mm]
&=&\int P^{[k,k+n_1]}_*\n d\m-\int\m_01_{Z_{1}}d\m\nonumber\\[2mm]
&=&\|P^{[k,k+n_1]}_*\n-\m_01_{Z_{1}}\|_1.
\end{eqnarray}
Therefore, let us denote
$$
\g_{1}=\|P^{[k,k+n_1]}_*\l-\m_k1_{Z_{1}}\|_1.
$$

One can see that
\begin{eqnarray}\label{1g}
1-\int\m_k1_{Z_{1}}d\m\geq 1-\int\m_kd\m\geq \frac{1}{2}.
\end{eqnarray}

Due to $\m(X\setminus Z_{1})<2\d_1$ from \eqref{Z} we have
$$
\frac{1}{2}\int\m_kd\m\geq \int\m_k 1_{X\setminus
Z_{1}}d\m=\int\m_kd\m-\int\m_k1_{Z_{1}}d\m
$$
which yields
$$
\int\m_k1_{Z_{1}}d\m\geq \frac{\|\m_k\|_1}{2}.
$$
Therefore, one finds
\begin{eqnarray}\label{2g}
 1-\int\m_k1_{Z_{1}}d\m\leq 1-\frac{\|\m_k\|_1}{2}.
\end{eqnarray}
Hence, from \eqref{1g},\eqref{2g} we infer
$$
\frac{1}{2}\leq \g_{1}\leq 1-\frac{\|\m_k\|_1}{2}
$$

Thus, at $n\geq k+n_1$ we obtain
\begin{eqnarray*}
\|P^{[k,n]}_*\l-P^{[k,n]}_*\n\|_1&=&\|P^{[k+n_1,n]}_*\big(P^{[k,k+n_1]}_*\l-\m_k1_{Z_{1}}\big)-
P^{[k+n_1,n]}_*\big(P^{[k,k+n_1]}_*\n-\m_k1_{Z_{1}}\big)\|_1\\[2mm]
&=&\g_{1}\|P^{[k+n_1,n]}_*\l_1-P^{[k+n_1,n]}_*\n_1\|_1,
\end{eqnarray*}
where
\begin{eqnarray*}
&&\l_1=\frac{1}{\g_{1}}\big(P^{[k,k+n_1]}_*\l-\m_k1_{Z_{1}}\big)\\[2mm]
&&\n_1=\frac{1}{\g_{1}}\big(P^{[k,k+n_1]}_*\n-\m_k1_{Z_{1}}\big).
\end{eqnarray*}
It is clear that $\l_1,\n_1\in\frak{M}$, so we have proved
\eqref{l-m} for $\ell=1$.

Now assume that \eqref{l-m} holds for $i=\ell$, i.e. there are
numbers $\{n_i\}_{i=1}^{\ell}$ such that for any $n\geq
K_\ell:=\sum\limits_{i=1}^\ell n_i$ one has
\begin{equation}\label{l-m1}
\|P^{[k,n]}_*\l-P^{[k,n]}_*\n\|_1=\bigg(\prod_{i=1}^{\ell}\g_{i}\bigg)
\|P^{[K_{\ell},n]}_*\l_\ell-P^{[K_\ell,n]}_*\n_\ell\|_1,
\end{equation}
where $\l_\ell,\n_\ell\in\frak{M}$, and
$$
\frac{1}{2}\leq \g_{i}\leq 1-\frac{\|\m_{K_{i-1}}\|_1}{2}, \ \
i=1,\dots,\ell.
$$

Let us prove \eqref{l-m} at $i=\ell+1$. According to condition
($C_1$) there is a positive measure $\m_{K_\ell}$. One can find
$\d_{\ell+1}>0$ such that for any $Z\in\cf$ with
$\m(Z)<2\d_{\ell+1}$ one has
\begin{equation}\label{Z}
\int\m_{K_\ell}1_{Z}d\m<\frac{\|\m_{K_\ell}\|_1}{2}.
\end{equation}

For $\l_\ell$ and $\n_\ell$ from condition ($C_1$) one finds
$X_{\ell+1},Y_{\ell+1}\subset\cf$ and $n_{\ell+1}\in\bn$ such that
one has
$$
\max\big\{\m(X\setminus X_{\ell+1}),\m(X\setminus
Y_{\ell+1})\big\}<\d_{\ell+1}
$$
and
\begin{equation*}
P^{[K_\ell,K_\ell+n_{\ell+1}]}_*\l_\ell\geq\m_{K_\ell}1_{X_{\ell+1}},
\ \
P^{[K_\ell,K_\ell+n_{\ell+1}]}_*\n_\ell\geq\m_{K_\ell}1_{Y_{\ell+1}}.
\end{equation*}
Denote $Z_{\ell+1}=X_{\ell+1}\cap Y_{\ell+1}$, then one can see that
$\m(X\setminus Z_{\ell+1})<2\d_{\ell+1}$ and
\begin{equation}\label{4C}
P^{[K_\ell,K_\ell+n_{\ell+1}]}_*\l_\ell\geq\m_{K_\ell}1_{Z_{\ell+1}},
\ \
P^{[K_\ell,K_\ell+n_{\ell+1}]}_*\n_\ell\geq\m_{K_\ell}1_{Z_{\ell+1}}.
\end{equation}

Denoting $K_{\ell+1}=K_{\ell}+n_{\ell+1}$, and similarly to
\eqref{l-z} we get

\begin{eqnarray*}
\|P^{[K_\ell,K_{\ell+1}]}_*\l_{\ell}-\m_{K_\ell} 1_{Z_{\ell+1}}\|_1
&=&
\|P^{[K_\ell,K_{\ell+1}]}_*\n_\ell-\m_{K_\ell}1_{Z_{\ell+1}}\|_1\\[2mm]
&=&1-\int\m_{K_\ell} 1_{Z_{\ell+1}}d\m
\end{eqnarray*}

Denote
$$
\g_{\ell+1}=1-\int\m_{K_\ell} 1_{Z_{\ell+1}}d\m,
$$
hence using $\m(X\setminus Z_{\ell+1})<2\d_{\ell+1}$ and the same
argument as \eqref{1g},\eqref{2g} one finds
$$
\frac{1}{2}\leq \g_{\ell+1}\leq 1-\frac{\|\m_{K_\ell}\|_1}{2}.
$$

Now at $n\geq K_{\ell+1}$ we get
\begin{eqnarray*}
\|P^{[K_\ell,n]}_*\l_\ell-P^{[K_\ell,n]}_*\n_\ell\|_1&=&
\big\|P^{[K_{\ell+1},n]}_*\big(P^{[K_\ell,K_{\ell+1}]}_*\l_\ell-\m_{K_\ell}1_{Z_{\ell+1}}\big)\\[2mm]
&&-
P^{[K_{\ell+1},n]}_*\big(P^{[K_\ell,K_{\ell+1}]}_*\n_\ell-\m_{K_\ell}1_{Z_{\ell+1}}\big)\big\|_1\\[2mm]
&=&\g_{\ell+1}\|P^{[K_{\ell+1},n]}_*\l_{\ell+1}-P^{[K_{\ell+1},n]}_*\n_{\ell+1}\|_1,
\end{eqnarray*}
where
\begin{eqnarray*}
&&\l_{\ell+1}=\frac{1}{\g_{\ell+1}}\big(P^{[K_\ell,K_{\ell+1}]}_*\l_\ell-\m_{K_\ell}
1_{Z_{\ell+1}}\big)\\[2mm]
&&\n_{\ell+1}=\frac{1}{\g_{\ell+1}}\big(P^{[K_\ell,K_{\ell+1}]}_*\n_\ell-\m_{K_\ell}
1_{Z_{\ell+1}}\big).
\end{eqnarray*}
It is clear that $\l_{\ell+1},\n_{\ell+1}\in\frak{M}$. Thus, taking
into account \eqref{l-m1} we derive the desired equality.
\end{proof}

Next theorem shows that condition ($C_1$) is equivalent to the
satisfaction of the $L_1$-weak ergodicity of NHDMP.

\begin{thm}\label{EP2} Let a NHDMP $P^{[k,n]}(x,A)$ be given on $(X,\cf,\m)$.
Then the following assertions are equivalent
\begin{enumerate}
\item[(i)] $P^{[k,n]}(x,A)$ satisfies condition ($C_1$) with
\begin{equation}\label{gg2}
\sum_{n=1}^\infty\|\m_{k_n}\|_1=\infty
\end{equation}
for any increasing subsequence $\{k_n\}$ of $\bn$.

\item[(ii)] $P^{[k,n]}(x,A)$ satisfies the $L_1$-weak ergodicity.

\end{enumerate}
\end{thm}

\begin{proof} (i)$\Rightarrow$ (ii).
Then due to Proposition \ref{EP1} there is a subsequence
$\{K_\ell\}$ such that
\begin{equation}\label{l-m2}
\|P^{[k,n]}_*\l-P^{[k,n]}_*\n\|_1=\bigg(\prod_{i=1}^{\ell}\g_{i}\bigg)
\|P^{[K_{\ell},n]}_*\l_\ell-P^{[K_\ell,n]}_*\n_\ell\|_1,
\end{equation}
where $\l_\ell,\n_\ell\in\frak{M}$. Now from \eqref{gg} one gets
\begin{equation*}
\|P^{[k,n]}_*\l-P^{[k,n]}_*\n\|_1\leq 2\prod_{i=1}^{\ell} \bigg(
1-\frac{\|\m_{K_{i-1}}\|_1}{2}\bigg)
\end{equation*}

According to \eqref{gg2} we get the desired assertion.

Now consider the implication (ii)$ \Rightarrow$ (i). Fix $1>\i>0$.
Then given $k\in\bn$ and $\l,\m_0\in\frak{M}$, (here $\m_0$ is
fixed) one has
\begin{equation*}
\|P^{[k,n]}_*\l-P^{[k,n]}_*\m_0\|_1\to 0 \ \ \textrm{as} \ \
n\to\infty.
\end{equation*}
Then there is a sequence $\{Y_n\}\subset\cf$ such that
$\m(X\setminus Y_n)\to 0$, as $n\to\infty$, and
$$
\|(P^{[k,n]}_*\l-P^{[k,n]}_*\m_0)1_{Y_n}\|_{\infty}\to 0 \ \
\textrm{as} \ \ n\to\infty.
$$
Therefore, there exists an $n_k\in\bn$ such that $\m(X\setminus
Y_{n_k})<\i$ and
\begin{equation}\label{EP3}
\|(P^{[k,k+n_k]}_*\l-P^{[k,k+n_k]}_*\m_0)1_{Y_{n_k}}\|_{\infty}<\frac{\i}{2}
\end{equation}
Now denote $\n_k=P^{[k,k+n_k]}_*\m_0$.  Hence, from \eqref{EP3} we
get
\begin{eqnarray*}
P^{[k,k+n_k]}_*\l&\geq& P^{[k,k+n_k]}_*\l 1_{Y_{n_k}}\\[2mm]
&\geq& \n_k1_{Y_{n_k}}-\frac{\i}{2} 1_{Y_{n_k}}\\[2mm]
&\geq &\m_k1_{Y_{n_k}},
\end{eqnarray*}
where
$$
\m_k=\frac{1}{2}\n_k1_{A_k}, \ \ A_k=\bigg\{x\in X:\ \n_k(x)\geq
\frac{\i}{2}\bigg\}.
$$

Since $\n_k$ is a probability measure, therefore, we have
$0<\|\m_k\|_1\leq 1/2$, so
$$
1-\frac{\|\m_k\|_1}{2}\geq\frac{3}{4}.
$$
Hence, this completes the proof.
\end{proof}

If one takes $n_k=k+1$ in condition $C_1$, then we get the following

\begin{cor}\label{EP21} Let $P^{[k,n]}(x,A)$ be a NHDMP on
$(X,\cf,\m)$. If for each $k\in\bz_+$ there exist a positive measure
$\m_k\in L^1$, $\|\m_k\|_1\neq 0$,  and for every $\d>0$ and
$\l\in\frak{M}$ one can find a set $X_k\in \cf$ with $\m(X\setminus
X_k)<\d$ such that
\begin{equation}\label{1C1}
P^{[k,k+1]}_*\l\geq\m_k1_{X_k},
\end{equation}
with
\begin{equation}\label{1gg2}
\sum_{n=1}^\infty\|\m_{n}\|_1=\infty
\end{equation}
then the NHDMP satisfies the $L_1$-weak ergodicity.
\end{cor}

Now let us consider a nonhomogeneous version of condition $(C_0)$.
Namely, a NHDMP $P^{[k,n]}(x,A)$ given on $(X,\cf,\m)$ is said to
satisfy \textit{condition ($C_2$)} if for each $k\in\bz_+$ there
exists a positive measure $\m_k\in L^1$, $\|\m_k\|_1\neq 0$, and for
every $\l\in\frak{M}$ one can find a sequence
$\{X^{(k)}_n\}\subset\cf$ with $\m(X\setminus X^{(k)}_n)\to 0$, as
$n\to\infty$, and $n_0(\l,k)\in\bn$ such that for all $n\geq
n_0(\l,k)$ one has
\begin{equation}\label{C}
P^{[k,n]}_*\l\geq\m_k1_{X^{(k)}_n};
\end{equation}

From Proposition \eqref{EP1} and Theorem \ref{EP2} we immediately
see that condition ($C_2$) with \eqref{gg2} is sufficient for the
$L_1$-weak ergodicity. One the other hand, if NHDMP becomes
homogeneous then condition ($C_2$) reduces to $C_0$, but in
\cite{SG} it has been proved that the last condition (i.e.
\eqref{AA}) is equivalent to the $L_1$-strong ergodicity of the
homogeneous process. Therefore, one can formulate the following:\\

{\bf Problem}. Is Condition ($C_2$) with \eqref{gg2} necessary for
the $L_1$-weak ergodicity?

\section{Applications}

In this section we provide some application of the main result for
concrete cases.

\subsection{Discrete case}
 Let us consider a countable state space NHDMP. Namely, let
$X=\bn$ and $\m$ be the Poisson measure. Then NHDMP can be given in
a form of stochastic matrices $\{p^{[k,n]}_{i,j}\}_{i,j\in\bn}$.

\begin{thm}\label{a1} Let $\{p^{[k,n]}_{i,j}\}_{i,j\in\bn}$ be a NHDMP. If
there exists a sequence $\{\l_n\}_{\in\bn}$, $0\leq\l_n\leq 1$
satisfying
\begin{equation}\label{ll}
\sum_{n=1}^\infty(1-\l_n)=\infty
\end{equation}
and such that for some sequence of states $\{n_k\}$
\begin{equation}\label{pl}
p^{[k-1,k]}_{i,n_{k}}\geq \l_k \ \ \textrm{for all}\ \ i,k\in\bn,
\end{equation}
then the NHDMP satisfies the $L_1$-weak ergodicity.
\end{thm}

\begin{proof} Now we show that the process satisfies the condition
($C_1$). Indeed, for each $k\in\bz_k$ we first define a measure
$\m^{(k)}$ on $X$ as follows:
\begin{equation*}
\m^{(k)}_i=\left\{
\begin{array}{ll}
\l_k, \ \  i=n_k\\[2mm]
0, \ \ \ \ i\neq n_k
\end{array}
\right.
\end{equation*}

It is clear that $\|\m^{(k)}\|_1\neq 0$. From \eqref{pl} it follows
that
\begin{equation}\label{pm}
p^{[k-1,k]}_{i,j}\geq \m^{(k)}_j, \ \ \textrm{for all}  \ \
i,j\in\bn.
\end{equation}

Now take any $\n\in\frak{M}$ and each $k\in\bz_+$ we put $X_k=X$,
then from \eqref{pm} one finds
$$
P^{[k-1,k]}_*\n\geq\m^{(k)}  \ \ \textrm{for all}  \ \ k\in\bn.
$$

Hence, the condition ($C_1$) is satisfied. So, taking into account
\eqref{ll}, from Corollary \ref{EP21} we get the desired assertion.
\end{proof}

We note that the proved theorem extends a result of \cite{GAM,P}.

{\sc Example.} Let us consider more concrete examples. Assume that
the transition probability $p_{ij}^{[k,k+1]}$ is defined by
\begin{equation}\label{11}
p_{ij}^{[k,k+1]}=q_{ij}^{(k)}\l_{k,j}+r_{k,i}\d_{ij}, \ \
i,j\in\bn,\ k\in\bn,
\end{equation}
here $\l_{k,j}$,$q_{ij}^{(k)}$, $r_{k,i}$ are positive numbers with
the following constrains
\begin{equation}\label{12}
\sum_{j=1}^\infty(q_{ij}^{(k)}\l_{k,j}+r_{k,i}\d_{ij})=1, \ \ \
\textrm{for all}  \ i\in\bn.
\end{equation}

It is clear that $p^{[k,k+1]}_{ik}\geq \l_{k,k}q_{i,k}^{(k)}$. Now
assume that
$$\inf\{q_{ik}^{(k)}: i\in\bn\}:=\gamma_k>0
$$
and
$$
\sum_{k=1}^\infty(1-\l_{k,k}\g_k)=\infty.
$$

Then one can see that $p^{[k,k+1]}_{ik}\geq \l_{k,k}\g_k$, this
means that conditions of Theorem \ref{a1} are satisfied with
$n_k=k$, $\l_k=\l_{k,k}\g_k$. Hence, the defined NHDMP is $L_1$-weak
ergodic.\\

Now consider more exact values of $\l_{k,j}$, $q_{ij}^{(k)}$,
$r_{k,i}$.

Define
\begin{equation}\label{13}
r_{k,i}=\frac{1}{k+i}, \ \ \ \l_{k,j}= \left\{
\begin{array}{lll}
0, \ \ 1\leq j \leq k-2 \ \ \textrm{or} \ \  j\geq k+1\\
\a_k, \ \ j=k-1\\
\b_{k}, \ \ j=k
\end{array}
\right. \end{equation} Note that $\a_k,\b_k$ will be chosen later
on.

Let $q_{ik}^{(k)}=\b_k$ for all $i\in\bn$, and $q_{ij}^{(k)}=0$ for
every $1\leq j\leq k-2$ and $j\geq k+1$. Now define
$q_{i,k-1}^{(k)}$ from the equality \eqref{12} as follows
$$
\a_kq_{i,k-1}^{(k)}+\b_k^2+r_{k,i}=1
$$
which implies that \begin{equation}\label{14}
q_{i,k-1}^{(k)}=\frac{1}{\a_k}\big(1-r_{k,i}-\b_k^2)
\end{equation}

Now choose $\a_k$ and $\b_k$ as follows
\begin{equation}\label{15}
\a_k=\frac{1}{k}, \ \ \b_k=\sqrt{\frac{k-1}{k}}, \ \ k\in\bn.
\end{equation}

Then from \eqref{13}-\eqref{15} one finds
$$
q_{i,k-1}^{(k)}=\frac{i}{k+i}.
$$

It is clear that $\g_k=\b_k$, therefore, from \eqref{13},\eqref{15}
one gets
$$
\sum_{k=1}^\infty(1-\l_{k,k}\g_k)=\sum_{k=1}^\infty(1-\b_k^2)=
\sum_{k=1}^\infty\frac{1}{k}=\infty.
$$

Hence, due to Theorem \ref{a1} the following NHDMP defined by
\begin{equation*}
p^{[k,k+1]}_{ij}= \left\{
\begin{array}{lll}
\frac{\d_{ij}}{k+i}, \  \qquad \quad\  \ \ \ \ \ \ 1\leq j \leq k-2
\ \ \textrm{or} \ \ j\geq
k+1,\\[2mm]
\frac{1}{k+i}\big(\frac{i}{k}+\d_{i,k-1}\big), \ \ j=k-1,\\[2mm]
\frac{k-1}{k}+\frac{1}{k+i}\d_{i,k}, \ \ \  \ j=k,
\end{array}
\right.
\end{equation*}
satisfies the $L_1$-weak ergodicity.

\subsection{Continuous case} Let $(X,\mathcal{F},\m)$ be a probability space and
$P^{[k,m]}(x,A)$ be a NHDMP on this space.

\begin{thm}\label{EPP} Let $P^{[k,m]}(x,A)$ be a NHDMP on $(X,\mathcal{F},\m)$. If
for every $k\in\bz_+$ there exists a set $A_k\in\mathcal{F}$ and a
number $\a_k>0$ such that
\begin{equation}\label{pa}
P^{[k-1,k]}(x,A_k)\geq\a_k \ \textrm{for all}\ \ x\in X,\ k\in\bn
\end{equation}
where
\begin{equation}\label{al}
\sum_{n=1}^\infty\bigg(1-\frac{\a_n}{2}\bigg)=\infty.
\end{equation}
Then the NHDMP satisfies the $L_1$-weak ergodicity.
\end{thm}

\begin{proof} To prove the statement it is enough to establish that the process
satisfies condition $C_1$. Indeed, for each $k\in\bz$ let us define
\begin{equation*}
\n_k(A)=\bigwedge_{x\in X}P^{[k-1,k]}(x,A\cap A_k), \ \ A\in\cf
\end{equation*}
Due to Theorem IV.7.5 \cite{Du} the defined mapping $\n_k$ is a
measure on $X$, and moreover, one has $\n_k(A_k)\geq\a_k$. Now put
$$
\m_k(A)=\frac{\n_k(A\cap A_k)}{\n_k(A_k)},  \ \ A\in\cf.
$$

Then one can see that
\begin{equation}\label{pa1}
P^{[k-1,k]}_*\d_x\geq\a_k\m_k \ \textrm{for all}\ \ x\in X, k\in\bn.
\end{equation}

It is clear that $\|\m_k\|_1\neq 0$.

Denote
$$\cm=\bigg\{\nu=\sum_{i=1}^n\alpha_i\delta_{x_i}: \
\sum_{i=1}^n\alpha_i=1, \alpha_i\geq0, \{x_i\}_{i=1}^n\subset X,\
n\in\mathbb{N}\bigg\}$$ which is convex set. Therefore, from
\eqref{pa1} we immediately find that
\begin{equation}\label{pa2}
P^{[k-1,k]}_*\m\geq\a_k\m_k \ \textrm{for all}\ \ \m\in\mathcal{M}.
\end{equation}

Due to the fact (see \cite{Du}) that the set $\mathcal{M}$ is a weak
dense subset of the set of all probability measures
$\widetilde{\frak{M}}$ on $(X,\mathcal{F})$, i.e.
$\overline{\cm}^w=\widetilde{\frak{M}}$. Hence, from \eqref{pa2} one
gets
\begin{equation}\label{pa3}
P^{[k-1,k]}_*\l\geq\a_k\m_k \ \textrm{for all}\ \
\l\in\widetilde{\frak{M}}.
\end{equation}

Now for each each $k\in\bz_+$ we put $X_k=X$, then from \eqref{pa3}
it follows condition $C_1$. So, taking into account \eqref{al}, from
Corollary \ref{EP21} we get the desired assertion.
\end{proof}

\subsection{Quadratic stochastic processes}

In this section we apply the obtained results to discrete time
quadratic stochastic processes. Note that such kind of processes
relate with quadratic operators as well as Markov processes with
linear operators (see \cite{GMR} for review).

Let $(X,\mathcal{F},\m)$ be a probability space. We recall that a
family of functions $\{Q^{[k,n]}(x, y, A)\}$ defined for $n>k$
($k,n\in\bz_+)$ for all $x,y\in X$, $A\in\mathcal{F}$, is called
{\it discrete quadratic stochastic process (DQSP)} if the following
conditions are satisfied:
\begin{enumerate}

\item[(i)]  $Q^{[k,n]}(x,y,A) = Q^{[k,n]}(y,x,A)$ for any $x, y\in X$ and
$A\in \mathcal{F}$;

\item[(ii)]$Q^{[k,n]}(x,y,\cdot)\in \frak{M}$ for any fixed $x, y\in X$;

\item[(iii)] $Q^{[k,n]}(x,y,A)$ as a function of $x$ and $y$ is measurable
on $(X\times X, {\mathcal{F}}\otimes{\mathcal{F}})$ for any $A\in
\mathcal{F}$;

\item[(iv)] (Analogue of the Chapman-Kolmogorov equation) for the
initial
 measure $\mu\in \frak{M}$ and arbitrary $k<m<n$, $k,m,n\in\bz_+$ we have either

 (iv)$_A$
 $$
Q^{[k,n]}(x,y,A) = \int_X\int_X
Q^{[k,m]}(x,y,du)Q^{[k,n]}(u,v,A)\mu_m(dv),$$ where the measure
$\mu_m$ on $(X,{\mathcal{F}})$ is defined by
$$\mu_m(B)=\int_X\int_X Q^{[0,m]}(x,y,B)\mu(dx)\mu(dy),$$
for any $B\in {\mathcal{ F}}$, or

(iv)$_B$ $$Q^{[k,n]}(x,y,A)= \int_X \int_X \int_X \int_X
Q^{[k,m]}(x,z,du)Q^{[k,m]}(y,v,dw)Q^{[m,n]}(u,w,A)\mu_k(dz)\mu_k(dw).$$
\end{enumerate}

If the condition $(iv)_A$ (resp. $(iv)_B$) holds, then DQSP is
called of {\it type (A)} (resp. (B)).

The process $Q^{[k,n]}(x,y,A)$ can be interpreted as the probability
of the following event: if $x$ and $y$ in $X$ interact at time $k$,
then one of the elements of the set $A\in \mathcal{F}$ will be
realized at time $n$. All phenomena in physics, chemistry, and
biology develop along non-zero finite time intervals. Therefore, we
assume that the maximum of these values of time is equal to 1.
Hence, $Q^{[k,n]}(x,y,A)$ is defined for $n-k\geq 1$ (we refer the
reader to \cite{GMR} for more information).

By $\frak{M}^2$ we denote the set of all probability measures on
$X\times X$ which are absolutely continuous w.r.t. $\m\otimes\mu$,
i.e. $\frak{M}^2$ can be considered as a subset of $L^1(X\times
X,{\mathcal{F}}\otimes{\mathcal{F}},\mu\otimes\mu)$. Given DQSP
$Q^{[k,n]}(x,y,A)$ one can define
\begin{eqnarray}\label{Q*}
&& (Q^{[k,n]}_*\tilde\n)(A)=\int_X\int_X Q^{[k,n]}(x,y,A)d\tilde\n
(x,y), \ \ \ \ \tilde\n\in L^1(X\times X,\mu\otimes\mu).
\end{eqnarray}

We recall that a DQSP $Q^{[k,n]}(x,y,A)$ is said to satisfy the
\textit{$L_1$-weak ergodicity} ( or {\it ergodic principle}) if for
any probability measures $\tilde\l,\tilde\nu\in\frak{M}^2$ and
$k\in\bz_+$ one has
$$\lim_{n\rightarrow\infty}\|Q^{[k,n]}_*\tilde\l-Q^{[k,n]}_*\tilde\nu\|_1=0;$$

Let $Q^{[k,n]}(x,y,A)$ be a given DQSP. Now define the following
transition probability
\begin{equation}\label{Q-P}
P_Q^{[k,n]}(x,A)=\int_X Q^{[k,n]}(x,y,A)d\m_k(y).
\end{equation}

In \cite{M} it has been proved the following

\begin{thm}\label{EP-PQ} Let $Q^{[k,n]}(x,y,A)$ be a given DQSP on
$(X,\mathcal{F},\mu)$. Then the following statements hold true:
\begin{enumerate}
\item[(i)] the defined $P_Q^{[k,n]}(x,A)$ is a NHDMP on
$(X,\mathcal{F},\mu)$;

\item[(ii)] the process $P_Q^{[k,n]}(x,A)$ satisfies the $L_1$-weak
ergodicity if and only if  $Q^{[k,n]}(x,y,A)$ satisfies the
$L_1$-weak ergodicity.
\end{enumerate}
\end{thm}

This theorem allows us to prove the following result.

\begin{thm}\label{WE-PQ} Let $Q^{[k,n]}(x,y,A)$ be a given DQSP on
$(X,\mathcal{F},\mu)$. If for every $k\in\bz_+$ there exists a set
$A_k\in\mathcal{F}$ and a number $\a_k>0$ such that
\begin{equation}\label{qpa1}
Q^{[k-1,k]}(x,y,A_k)\geq\a_k \ \textrm{for all}\ \ x,y\in X,\
k\in\bn
\end{equation}
where
\begin{equation}\label{qal}
\sum_{n=1}^\infty\bigg(1-\frac{\a_n}{2}\bigg)=\infty,
\end{equation}
then the DQSP is $L_1$-weak ergodic.
\end{thm}

\begin{proof} Consider the process $P_Q^{[k,n]}(x,A)$. Then from \eqref{Q-P} and
\eqref{qpa1} one finds
\begin{equation*}\label{qpa2}
P_Q^{[k-1,k]}(x,A_k)=\int_X Q^{[k-1,k]}(x,y,A_k)d\m_k(y)\geq\a_k \
\textrm{for all}\ \ x\in X,\ k\in\bn.
\end{equation*}
Hence, the Markov process $P_Q^{[k,n]}(x,A)$ satisfies the
conditions of Theorem \ref{EPP}, so it is $L_1$-weak ergodic.
Therefore, Theorem \ref{EP-PQ} implies the $L_1$-weak ergodicity of
$Q^{[k,n]}(x,y,A)$.
\end{proof}

Note that the last theorem improves the result of \cite{SGa}.

\section*{Acknowledgments} The author
acknowledges the MOHE grant FRGS11-022-0170 and the Junior Associate
scheme of the Abdus Salam International Centre for Theoretical
Physics, Trieste, Italy.

\end{document}